\documentclass[12pt]{amsart}

\usepackage{etoolbox}
\makeatletter
\def\@secnumfont{\mdseries}
\def\section{\@startsection{section}{1}%
 \z@{.7\linespacing\@plus\linespacing}{.5\linespacing}%
 {\normalfont\scshape\centering}}
\def\subsection{\@startsection{subsection}{2}%
 \z@{.5\linespacing\@plus.7\linespacing}{-.5em}%
 {\normalfont\bfseries}}

\patchcmd{\@thm}{\let\thm@indent\indent}{\let\thm@indent\noindent}{}{}
\patchcmd{\@thm}{\thm@headfont{\scshape}}{\thm@headfont{\bfseries}}{}{}

\makeatother

\usepackage{enumerate}
\usepackage[dvips]{graphicx}
\usepackage{hyperref,color,mathdots,accents,extarrows,bigints}
\usepackage{stmaryrd}
\usepackage[all]{xypic}
\usepackage[ansinew]{inputenc}
\usepackage{amsmath,amsfonts,amssymb}
\usepackage[all,graph]{xy}
\usepackage[shortcuts]{extdash}

\theoremstyle{plain}
\newtheorem{theorem}{Theorem}[section]
\newtheorem{lemma}[theorem]{Lemma}

\newtheorem{corollary}[theorem]{Corollary}

\newtheorem*{proposition*}{Proposition}

\newtheorem*{question*}{Question}

\theoremstyle{remark}

\newtheorem*{definition}{Definition}

\newtheorem*{example}{Example}

\makeatletter
\def\@tocline#1#2#3#4#5#6#7{\relax
 \ifnum #1>\c@tocdepth 
 \else
 \par \addpenalty\@secpenalty\addvspace{#2}%
 \begingroup \hyphenpenalty\@M
 \@ifempty{#4}{%
 \@tempdima\csname r@tocindent\number#1\endcsname\relax
 }{%
 \@tempdima#4\relax
 }%
 \parindent\z@ \leftskip#3\relax \advance\leftskip\@tempdima\relax
 \rightskip\@pnumwidth plus4em \parfillskip-\@pnumwidth
 #5\leavevmode\hskip-\@tempdima
 \ifcase #1
 \or\or \hskip 1em \or \hskip 2em \else \hskip 3em \fi%
 #6\nobreak\relax
 \dotfill\hbox to\@pnumwidth{\@tocpagenum{#7}}\par
 \nobreak
 \endgroup
 \fi}
\makeatother

\def\Q{\mathbb{Q}}

\def\Z{\mathbb{Z}}

\def\N{\mathbb{N}}

\newcommand{\mraisebox}[2]{\mbox{\raisebox{#1}{$#2$}}}

\newcommand{\tmfrac}[2]{\mbox{\large$\frac{#1}{#2}$}}

\DeclareMathAlphabet{\mathbf}{OML}{cmm}{b}{it}

\def\bnm{\begin{enumerate}}
\def\enm{\end{enumerate}}
\def\ba{\begin{array}}
\def\ea{\end{array}}
\def\bpp{\begin{pmatrix}}
\def\epp{\end{pmatrix}}

\def\hom{\operatorname{Hom}}

\def\op{\operatorname}

\numberwithin{equation}{section}

\begin{document}

\title{Linking forms of amphichiral knots}

\author{Stefan Friedl}
\address{Fakult\"at f\"ur Mathematik\\ Universit\"at Regensburg\\   Germany}
\email{sfriedl@gmail.com}

\author{Allison N.~Miller}
\address{Department of Mathematics, University of Texas, Austin, USA}
\email{amiller@math.utexas.edu}

\author{Mark Powell}
\address{
D\'epartement de Math\'ematiques, Universit\'e du Qu\'ebec \`a Montr\'eal, Canada}
\email{mark@cirget.ca}


\def\subjclassname{\textup{2010} Mathematics Subject Classification}
\expandafter\let\csname subjclassname@1991\endcsname=\subjclassname
\expandafter\let\csname subjclassname@2000\endcsname=\subjclassname
\subjclass{%
 57M25, 
 57M27, 
}

\begin{abstract}
We give a simple obstruction for a knot to be amphichiral, in terms of the homology of the $2$-fold branched cover. We work with unoriented knots, and so obstruct both positive and negative amphichirality.
\end{abstract}

\maketitle

\section{Introduction}
By a knot we mean a $1$-dimensional submanifold of $S^3$ that is diffeomorphic to $S^1$. Given a knot $K$ we denote  its \emph{mirror image} by $mK$, the image of $K$ under an orientation reversing homeomorphism $S^3 \to S^3$.  We say that a knot $K$ is \emph{amphichiral} if $K$ is (smoothly) isotopic to $mK$. Note that we consider unoriented knots, so we do not distinguish between positive and negative amphichiral knots.

In this paper we will see that the homology of the $2$-fold branched cover can be used to show that many  knots are not amphichiral. Before we state our main result we recall some definitions and basic facts.
\bnm
\item \label{item:intro-list-1} Given a knot $K$ we denote the 2-fold cover $\Sigma(K)$ of $S^3$ branched along $K$ by $\Sigma(K)$. If $A$ is a  Seifert matrix for $K$, then a presentation matrix for $H_1(\Sigma(K);\Z)$ is given by $A+A^T$.  See~\cite{Ro90} for details.
\item
The \emph{determinant of $K$} is defined as the order of $H_1(\Sigma(K);\Z)$.
By (\ref{item:intro-list-1}) we have  $\det(K) = \det(A+A^T)$. Alternative definitions are given by $\det(K)=\Delta_K(-1)=J_K(-1)$ where $\Delta_K(t)$ denotes the Alexander polynomial and $J_K(q)$ denotes the Jones polynomial~\cite[Corollary~9.2]{Li97},~\cite[Theorem~8.4.2]{Ka96}.
\item Given an abelian additive group $G$ and a prime $p$,
the \emph{$p$-primary part of $G$} is defined as
\[ \hspace{1cm} G_p\,:=\,\{g\in G\,|\, p^k\cdot g=0\mbox{ for some }k\in \N_0\}.\]
\enm
By studying the linking form on the $2$-fold branched cover we
prove the following theorem which is the main result of this article.

\begin{theorem}\label{mainthm}
Let $K$ be a knot and let $p$ be a prime with $p\equiv 3\mod 4$.
If $K$ is amphichiral, then the $p$-primary part of $H_1(\Sigma(K))$ is either zero or it is not cyclic.
\end{theorem}

The following corollary, proven by Goeritz~\cite[p.~654]{Go33} in 1933, gives an even more elementary obstruction for a knot to be amphichiral.

\begin{corollary}\label{maincor}\textbf{\emph{(Goeritz)}}
Suppose $K$ is an amphichiral knot and $p$ is a prime with $p\equiv 3\mod{4}$. Then either $p$ does not divide $\det(K)$ or $p^2$ divides $\det(K)$.
\end{corollary}

In fact Goeritz showed an even stronger statement: given such $p$ the maximal power of $p$ that divides $\det(K)$ is even.
This elegant and rather effective result of Goeritz also appears in Reidemeister's classic textbook on knot theory \cite[p.~30]{Re74} from the 1930s, but it did not appear in any of the more modern textbooks.

In the following we present some examples which show the strength of the Goeritz theorem and we also show that our Theorem~\ref{mainthm} is independent of the Goeritz theorem.

\begin{example}\mbox{}
\bnm[(i)]
\item For the trefoil $3_1$ we have $\det(3_1)=3$, so Corollary~\ref{maincor} immediately implies the very well known fact that the trefoil is not amphichiral. In most modern accounts of knot theory one uses  either the signature of a knot or the Jones polynomial applied to prove that the trefoil is not amphichiral.  We think that it is worth recalling that even the determinant can be used to prove this statement.
\item As a reality check, consider the figure eight knot $4_1$, which is amphichiral~\cite[p.~17]{BZH14}. We have $\det(4_1)=5$, which is consistent with Corollary~\ref{maincor}.
\item  Corollary~\ref{maincor} also provides information on occasions when many other, supposedly more powerful invariants,  fail.
For example the chirality of the knot $K=10_{71}$ is difficult to detect, since the Tristram-Levine signature function~\cite{Le69,Tr69} of $K$ is identically zero and the HOMFLY and Kauffman polynomials of $K$ do not detect chirality. In fact in~\cite{RGK94} Chern-Simons invariants were used to show that $10_{71}$ is not amphichiral. However, a quick look at Knotinfo~\cite{CL} shows that $\det(10_{71})=77=7\cdot 11$, i.e.\ $10_{71}$  does not satisfy the criterion from Corollary~\ref{maincor} and so we see that this knot is not amphichiral.
\item On the other hand it is also quite easy to find a knot for which Corollary~\ref{maincor}, and also Theorem~\ref{mainthm} below, fail to show that it is not amphichiral. For example, the torus knot $T(5,2)=5_1$ has determinant $\det(5_1)= 5$, but is not amphichiral since its signature is nonzero.
\item
One quickly finds examples of knots where Theorem~\ref{mainthm} detects chirality, but Corollary~\ref{maincor} fails to do so.
Consider the Stevedore's knot $6_1$. A Seifert matrix is given by $A=\mraisebox{0.05cm}{\hspace{-0.05cm}\scriptsize{\scalebox{0.83}{\Big(}\ba{rr} 1&\phantom{-}0\\ 1&-2\ea\scalebox{0.83}{\Big)}}\hspace{-0.05cm}}$. By the aforementioned formula a presentation matrix for $H_1(\Sigma(K))$ is given by $A+A^T$. It is straightforward to compute that $H_1(\Sigma(K))\cong \Z_9$. So it follows from Theorem~\ref{mainthm}, applied with $p=3$, that the Stevedore's knot is not amphichiral.
\enm
\end{example}

We refer to  \cite{Ha80}, \cite[Proposition~1]{CM83} and~\cite[Theorem~9.4]{Hi12}
for results on Alexander polynomials of amphichiral knots. In principle
these results give stronger obstructions to a knot being amphichiral than the Goeritz Theorem, but in practice it seems to us that these obstructions are fairly hard to implement.

To the best of our knowledge Theorem~\ref{mainthm} is the first obstruction to a knot being amphichiral that uses the structure of the homology module.
We could not find a result in the literature that implies Theorem~\ref{mainthm}, essentially because the previous results on Alexander polynomials mentioned above did not consider the structure of the Alexander module. Similarly Goeritz~\cite{Go33} did not consider the structure of the first homology for the $2$-fold branched cover.  So to the best of our knowledge, and to our surprise, Theorem~\ref{mainthm} seems to be new.  Notwithstanding, the selling point of our theorem is not that it yields any new information on amphichirality, but that the obstruction is very fast to compute and frequently effective.

As mentioned above, the proof of Theorem~\ref{mainthm} relies on the study of the linking form on the $2$-fold branched cover $\Sigma(K)$. Similarly, as in \cite[Theorem~9.3]{Hi12}, one can use the Blanchfield form~\cite{Bl57} to obtain restrictions on the primary parts of the Alexander module. Moreover one can use twisted Blanchfield forms~\cite{Po16} to obtain conditions on twisted Alexander polynomials~\cite{Wa94,FV10} and twisted Alexander modules of amphichiral knots. Our initial idea had been to use the latter invariant. But we quickly found that even the elementary invariants studied in this paper are fairly successful.  To keep the paper short we refrain from discussing these generalisations.

The paper is organised as follows. Linking forms and basic facts about them are recalled in Section~\ref{section:linking-forms}. The proofs of Theorem~\ref{mainthm} and Corollary~\ref{maincor} are given in Section~\ref{section-proofs}.

\subsection*{Acknowledgments.}
We are indebted to Cameron Gordon, Chuck Livingston and Loren\-zo Traldi for pointing out the paper of Goeritz~\cite{Go33}, which contains Corollary~\ref{maincor}. Despite a fair amount of literature searching before we posted the first version, we did not come across this paper. We are pleased that the beautiful result of Goeritz was brought back from oblivion.

The authors are grateful to the Hausdorff Institute for Mathematics in Bonn, in whose excellent research atmosphere this paper was born. We are also grateful to Jae Choon Cha and Chuck Livingston~\cite{CL} for providing Knotinfo, which is an indispensable tool for studying small crossing knots.

The first author acknowledges the support provided by the SFB 1085 `Higher
Invariants' at the University of Regensburg, funded by the DFG.
The third author is supported by an NSERC Discovery Grant.

\section{Linking forms}\label{section:linking-forms}

\begin{definition}
\mbox{}
\begin{enumerate}[(i)]
\item
A \emph{linking form} on a finitely generated abelian group $H$ is a map $\lambda\colon H\times H\to \Q/\Z$ which has the following properties:
\bnm
\item $\lambda$ is bilinear and symmetric,
\item $\lambda$ is nonsingular, that is the adjoint map $H\to \hom(H,\Q/\Z)$ given by $a\mapsto (b\mapsto \lambda(a,b))$ is an isomorphism.
\enm
\item Given a linking form $\lambda\colon H\times H\to \Q/\Z$ we denote the linking form on $H$ given by $(-\lambda)(a,b)=-\lambda(a,b)$ by $-\lambda$.
\end{enumerate}
\end{definition}

\begin{lemma}\label{lem:linking-forms-on-cyclic-modules}
Let $p$ be a prime and $n\in \N_0$. Every linking form $\lambda$ on $\Z_{p^n}$ is given by
\[ \ba{rcl} \Z_{p^n}\times \Z_{p^n}&\to & \Q/\Z\\
(a,b)&\mapsto &\lambda(a,b)=\tmfrac{k}{p^n}a\cdot b\in \Q/\Z\ea\]
for some $k\in \Z$ that is coprime to $p$.
\end{lemma}

\begin{proof}
Pick $k\in \Z$ such that $\frac{k}{p^n}=\lambda(1,1)\in \Q/\Z$. By the bilinearity we have $\lambda(a,b)=\frac{k}{p^n}a\cdot b\in \Q/\Z$ for all $a,b\in \Z_{p^n}$. It follows easily from the fact that $\lambda$ is nonsingular that $k$ needs to be coprime to $p$. To wit, if $p|k$ then $k=p\cdot k'$, so
for any $a\in p^{n-1}\Z_{p^n}$ and $b\in \Z_{p^n}$ we have $\lambda(a,b)= k' apb/p^{n-1}=0\in \Q/\Z$.  Therefore the non-trivial subgroup $p^{n-1}\Z_{p^n}$ lies in the kernel of the adjoint map, so the adjoint map is not injective.
\end{proof}

\noindent We recall the following well known lemma.

\begin{lemma}\label{lem:split-off-p-primary-summand}
Let $\lambda\colon H\times H\to \Q/\Z$ be a linking form and let $p$ be a prime. The restriction of $\lambda$ to the $p$-primary part  $H_p$ of $H$ is also nonsingular.
\end{lemma}

\begin{proof}
It suffices to show that there exists an orthogonal decomposition $H=H_p\oplus H'$.
Since $H$ is the direct sum of its $p$-primary subgroups we only need to show that if $p,q$ are two different primes and if $a\in H_p$ and $b\in H_q$, then $\lambda(a,b)=0$. So let $p$ and $q$ be two distinct primes. Since $p$ and $q$ are coprime there exist $x,y\in \Z$ with $px+qy=1$. It follows that $\lambda(a,b)=\lambda((px+qy)a,b)=\lambda(pxa,b)+\lambda(a,qyb)=0$.
\end{proof}

Let $\Sigma$ be an oriented rational homology 3-sphere, i.e.\ $\Sigma$ is a 3-manifold with $H_*(\Sigma;\Q)\cong H_*(S^3;\Q)$. Consider the maps
\[ H_1(\Sigma;\Z)\,\,\xrightarrow{\,\op{PD}^{-1}\,}\,\, H^2(\Sigma;\Z)\,\,\xleftarrow{\,\delta\,}\,\, H^1(\Sigma;\Q/\Z)\,\,\xrightarrow{\,\op{ev}\,}\,\, \hom(H_1(\Sigma;\Z),\Q/\Z),\]
where the maps are given as follows:
\bnm
\item the first map is given by the inverse of Poincar\'e duality, that is the inverse of the map given by capping with the fundamental class of the \emph{oriented} manifold $\Sigma$;
\item   the second map is the connecting homomorphism in the long exact sequence in cohomology corresponding to the short exact sequence $0\to \Z\to \Q\to \Q/\Z\to 0$ of coefficients; and
\item  the third map is the evaluation map.
\enm

The first map is an isomorphism by Poincar\'e duality, the second map is an isomorphism since $\Sigma$ is a rational homology sphere and so $H^i(\Sigma;\Q)=H_{3-i}(\Sigma;\Q)=0$ for $i=1,2$, and the third map is an isomorphism by the universal coefficient theorem, and the fact that $\Q/\Z$ is an injective $\Z$-module. Denote the corresponding isomorphism by $\Phi_\Sigma \colon H_1(\Sigma;\Z)\to \hom(H_1(\Sigma;\Z),\Q/\Z)$ and define
\[ \ba{rcl}\lambda_\Sigma\colon H_1(\Sigma;\Z)\times H_1(\Sigma;\Z)&\to & \Q/\Z\\
(a,b)&\mapsto & (\Phi_\Sigma(a))(b).\ea\]

\begin{lemma}
For every oriented rational homology 3-sphere $\Sigma$, the map $\lambda_\Sigma$ is a linking form.
\end{lemma}

\begin{proof}
We already explained why $\Phi_\Sigma$ is an isomorphism, which is equivalent to the statement that $\lambda_\Sigma$ is nonsingular. Seifet~\cite[p.~814]{Se33} gave a slightly non-rigorous proof that  $\lambda_\Sigma$ is symmetric, more modern proofs are given in \cite{Po16} or alternatively in \cite[Chapter~48.3]{Fr17}.
\end{proof}

\noindent The following lemma is an immediate consequence of the definitions and the obvious fact that for an oriented manifold $M$ we have $[-M]=-[M]$.

\begin{lemma}\label{lem:linking-form-reverse-orientation}
Let $\Sigma$ be an oriented rational homology 3-sphere. We denote the same manifold but with the opposite orientation by $-\Sigma$. For any $a,b\in H_1(\Sigma;\Z)=H_1(-\Sigma;\Z)$ we have
\[ \lambda_{-\Sigma}(a,b)\,\,=\,\,-\lambda_{\Sigma}(a,b).\]
\end{lemma}

Let $K\subset S^3$ be a knot and let $\Sigma(K)$ be the 2-fold cover of $S^3$ branched along $K$. Note that $\Sigma(K)$ admits a unique orientation such that the projection $p\colon \Sigma(K)\to S^3$ is orientation-preserving outside of the branch locus $p^{-1}(K)$. Henceforth we will always view $\Sigma(K)$ as an oriented manifold.

\begin{lemma}\label{lem:branched-covers-knots}\mbox{}
\bnm[(i)] 
\item\label{item:lem:branched-covers-item-1} Let $K$ and $J$ be two knots. If $K$ and $J$ are (smoothly) isotopic, then there exists an orientation-preserving diffeomorphism between $\Sigma(K)$ and $\Sigma(J)$.
\item\label{item:lem:branched-covers-item-2} Let $K$ be a knot. There exists an orientation-reversing diffeomorphism $\Sigma(K)\to \Sigma(mK)$.
\enm
\end{lemma}

\begin{proof}
The first statement follows immediately from the isotopy extension theorem~\cite[Theorem~II.5.2]{Ko93}. The second statement is an immediate consequence of the definitions.
\end{proof}

\section{Proofs}\label{section-proofs}

\subsection{Proof of Theorem~\ref{mainthm}}
Let $K$ be a knot and let $p$ be a prime.
By Lemma~\ref{lem:branched-covers-knots}~(\ref{item:lem:branched-covers-item-2}) there exists an orientation-preserving diffeomorphism $f\colon \Sigma(K)\to -\Sigma(mK)$
which means that $f$ induces an isometry
from $\lambda_{\Sigma(K)}$ to $\lambda_{-\Sigma(mK)}$.
It follows from Lemma~\ref{lem:linking-form-reverse-orientation}
that $f$ induces an isometry
from the linking form $\lambda_{\Sigma(K)}$ to $-\lambda_{\Sigma(mK)}$.
In particular $f$ induces an isomorphism $H_1(\Sigma(K);\Z)_p\to H_1(\Sigma(mK);\Z)_p$ between the $p$-primary parts of the underlying abelian groups.

Now suppose that $K$ is amphichiral, meaning that $mK$ is isotopic to $K$.
Write $H=H_1(\Sigma(K))$, denote the $p$-primary part of $H$ by $H_p$, and let $\lambda_p\colon H_p\times H_p\to \Q/\Z$ be the restriction of the linking form $\lambda_{\Sigma(K)}$ to $H_p$.
It follows from Lemma~\ref{lem:split-off-p-primary-summand} that $\lambda_p$ is also a linking form.
Then by Lemma~\ref{lem:branched-covers-knots} (\ref{item:lem:branched-covers-item-2}) and the above discussion, there exists an isometry
\[\Phi\colon (H_p,-\lambda_p)\xrightarrow{\cong} (H_p,\lambda_p).\]

Now suppose that $H_p$ is cyclic and nonzero, so that we can make the identification $H_p=\Z_{p^n}$ for some $n\in \N$.
By Lemma~\ref{lem:linking-forms-on-cyclic-modules}, there exists a $k\in \Z$, coprime to $p$, such that $\lambda_p(a,b)=\frac{k}{p}ab\in \Q/\Z$ for all $a,b\in \Z_{p^n}$.

The isomorphism $\Phi\colon \Z_{p^n}\to \Z_{p^n}$ is given by multiplication by some $r\in \Z$ that is coprime to $p$. We have that
\[-\tmfrac{k}{p^n}\,\,=\,\,(-\lambda_p)(1,1)\,\,=\,\,\lambda_p(r\cdot 1,r\cdot 1)\,\,=\,\,\tmfrac{k}{p^n}r^2\,\,\in\,\,\Q/\Z.\]
Thus there exists $m\in \Z$ such that $-\tmfrac{k}{p^n} = \tmfrac{k}{p^n}r^2 +m \in \Z$, so $-k=kr^2 +p^n m$.  Working modulo $p$ we obtain $-k \equiv kr^2 \mod{p}$.  Since $k$ is coprime to $p$, it follows that $-1\equiv r^2\mod p$.

But it is a well known fact from classical number theory, see e.g.\ \cite[p.~133]{Co09}, that for an odd prime $p$ the number $-1$ is a square mod $p$ if and only if $p\equiv 1\mod{4}$. Thus we have shown that for an amphichiral knot, and a prime $p$ such that $H_p$ is nontrivial and cyclic, we have that $p\equiv 1\mod 4$. This concludes the proof of (the contrapositive of) Theorem~\ref{mainthm}.

\subsection{The proof of Corollary~\ref{maincor}}

Let $K$ be an amphichiral knot. Recall that by definition of the determinant we have $\det(K)=|H_1(\Sigma(K))|$. Now let $p$ be a prime with $p\equiv 3\mod 4$ that divides $\det(K)$. We have to show that $p^2$ divides $\det(K)$.
Denote the $p$-primary part of $H_1(\Sigma(K))$ by $H_p$. Since $p$ divides $\det(K)$, we see that $H_p$ is nonzero.
By Theorem~\ref{mainthm}, we know that $H_p$ is not cyclic.
But this implies that $p^2$ divides the order of $H_p$, which in turn implies that $p^2$ divides $\det(K)$.
This concludes the proof of Corollary~\ref{maincor}.


\begin{thebibliography}{10}
\bibitem[Bl57]{Bl57}
R. Blanchfield, {\em Intersection theory of manifolds with operators with
              applications to knot theory}, Ann. of Math. (2) 65 (1957),
    {340--356}.
\bibitem[BZH14]{BZH14}
G. Burde, H. Zieschang and M. Heusener,
{\em Knots}, 3rd fully revised and extended edition.
De Gruyter Studies in Mathematics 5 (2014).
\bibitem[CL]{CL}
J. C. Cha and C. Livingston, {\em Knotinfo},
\url{http://www.indiana.edu/~knotinfo/}.
\bibitem[Co09]{Co09}
W. Coppel, {\em Number theory. An introduction to mathematics},  2nd ed.
Universitext (2009).
\bibitem[CM83]{CM83}
D. Coray and F. Michel, {\em  Knot cobordism and amphicheirality},
Comment. Math. Helvetici 58 (1983), 601--616.
\bibitem[Fr17]{Fr17}
S. Friedl, {\em Algebraic topology}, lecture notes, University of Regensburg  (2017)\\
\url{www.uni-regensburg.de/Fakultaeten/nat_Fak_I/friedl/papers/friedl_topology.pdf}
\bibitem[FV10]{FV10}
S. Friedl and  S. Vidussi,
 {\em A survey of twisted Alexander polynomials}, The Mathematics of Knots: Theory and Application, editors: Markus Banagl and Denis Vogel (2010), 45--94.
\bibitem[Go33]{Go33}
 L. Goeritz, {\em Knoten und quadratische Formen}, Math.\ Zeit. 36 (1933), 655--676.
\bibitem[HK79]{HK79}
R. Hartley and A. Kawauchi, {\em Polynomials of amphicheiral knots},
Math. Ann. 243 (1979), 63--70.
\bibitem[Ha80]{Ha80}
R. Hartley {\em Invertible amphicheiral knots},
   Math. Ann. 252 (1980), No. 2, 103--109.
\bibitem[Hi12]{Hi12}
J. Hillman, {\em Algebraic invariants of links}, second edition,
  Series on Knots and Everything 52, World Scientific Publishing (2012).
\bibitem[Ka96]{Ka96}
A. Kawauchi, {\em A survey of knot theory}, Birkh\"auser (1996).
\bibitem[Ko93]{Ko93}
A. Kosinski, {\em Differentiable manifolds}, Pure and Applied Mathematics, 138. Academic Press.  (1993).
\bibitem[Le69]{Le69}
J. Levine, {\em Invariants of knot cobordism},
Invent. Math. 8 (1969), 98--110.
\bibitem[Li97]{Li97}
W. B. R. Lickorish. {\em An introduction to knot theory}, Graduate Texts in Mathematics 175, Springer Verlag (1997).
\bibitem[Po16]{Po16}
M. Powell, {\em Twisted Blanchfield pairings and symmetric chain complexes}, Quarterly Journal of Mathematics 67 (2016),  715--742.
\bibitem[RGK94]{RGK94}
P. Ramadevi, T.  Govindarajan and R. Kaul, {\em Chirality of knots {$9_{42}$} and {$10_{71}$} and
              {C}hern-{S}imons theory}, Modern Phys. Lett. A 9 (1994), 3205--3217.
\bibitem[Re74]{Re74}
K. Reidemeister, {\em Knotentheorie},
Berlin-Heidelberg-New York: Springer-Verlag. VI (1974).
\bibitem[Ro90]{Ro90}
D. Rolfsen. {\em Knots and links}, Mathematics Lecture Series. 7. Houston, TX: Publish or Perish. (1990).
\bibitem[Se33]{Se33}
H. Seifert. {\em Verschlingungsinvarianten},
Sitzungsber. Preu\ss. Akad. Wiss., Phys.-Math. Kl. 1933, No.26-29 (1933), 811--828.
\bibitem[Tr69]{Tr69}
A.  Tristram, {\em  Some cobordism invariants for links},
Proc. Cambridge Philos. Soc. 66 (1969), 251--264.
\bibitem[Wa94]{Wa94}
M. Wada, {\em Twisted Alexander polynomial for finitely presentable groups}, Topology 33, no. 2 (1994), 241--256.
\end{thebibliography}
\end{document}